\documentclass[12pt]{amsart}
\usepackage{amssymb}
\usepackage{amsfonts}
\usepackage{amssymb,latexsym}
\usepackage{enumerate}
\usepackage{mathrsfs}
\usepackage{hyperref}
\usepackage{color}
\usepackage{tikz}

\makeatletter
\@namedef{subjclassname@2010}{%
  \textup{2010} Mathematics Subject Classification}
\makeatother

  [2002/01/19 v2.2g %
    AMS font definitions%
  ]
\DeclareFontFamily{U}{euf}{}
\DeclareFontShape{U}{euf}{m}{n}{%
  <5><6><7><8><9>gen*eufm%
  <10><10.95><12><14.4><17.28><20.74><24.88>eufm10%
  }{}
\DeclareFontShape{U}{euf}{b}{n}{%
  <5><6><7><8><9>gen*eufb%
  <10><10.95><12><14.4><17.28><20.74><24.88>eufb10%
  }{}

  [2002/01/19 v2.2g %
    AMS font definitions%
  ]
\DeclareFontFamily{U}{msb}{}
\DeclareFontShape{U}{msb}{m}{n}{%
  <5><6><7><8><9>gen*msbm%
  <10><10.95><12><14.4><17.28><20.74><24.88>msbm10%
  }{}

  [2002/01/19 v2.2g %
    AMS font definitions%
  ]
\DeclareFontFamily{U}{msa}{}
\DeclareFontShape{U}{msa}{m}{n}{%
  <5><6><7><8><9>gen*msam%
  <10><10.95><12><14.4><17.28><20.74><24.88>msam10%
  }{}

\newcommand{\calD}{\mathcal{D}}

\newcommand{\mC}{\mathbb{C}}

\newcommand{\mN}{\mathbb{N}}

\newcommand{\mR}{\mathbb{R}}

\newcommand{\mZ}{\mathbb{Z}}

\newtheorem{theorem}{Theorem}[section]

\newtheorem{proposition}[theorem]{Proposition}

\theoremstyle{definition}
\newtheorem{definition}[theorem]{Definition}
\newtheorem{example}[theorem]{Example}

\newtheorem{remark}[theorem]{Remark}

\numberwithin{equation}{section} 

\begin{document}

\title[Summation from the viewpoint of distributions]
{Some new examples of summation of divergent series from the viewpoint of distributions}



\author{Su Hu}
\address{Department of Mathematics, South China University of Technology, Guangzhou, Guangdong 510640, China}
\email{mahusu@scut.edu.cn}

\author{Min-Soo Kim}
\address{Department of Mathematics Education, Kyungnam University, Changwon, Gyeongnam 51767, Republic of Korea}
\email{mskim@kyungnam.ac.kr}



\subjclass[2010]{46F12, 11B68, 11M06}
\keywords{Infinite series, distributions, Fourier series, Bernoulli numbers, Apostol--Bernoulli numbers}

\begin{abstract}
Let $\{a_{1}, a_{2},\ldots, a_{n},\ldots\}$ be a sequence of complex numbers  which has at most polynomial growth and satisfies an extra assumption.
 In this paper, inspired  by a recent work of Sasane, we give an explanation of the sum
$$a_{1}+2a_{2}+3a_{3}+\cdots+na_{n}+\cdots,$$
and more generally, for any $k\in\mathbb{N},$ the sum
$$1^{k}a_{1}+2^{k}a_{2}+3^{k}a_{3}+\cdots+n^{k}a_{n}+\cdots,$$
from the viewpoint of distributions.
As applications, we explain the following summation formulas
\begin{equation*}
\begin{aligned}
1^{k}-2^{k}+3^{k}-\cdots&=-\frac{E_{k}(0)}{2},
\\
1^{k}+2^{k}+3^{k}+\cdots&=-\frac{B_{k+1}}{k+1},
\\
\epsilon^{1}1^{k}+\epsilon^{2}2^{k}+\epsilon^{3}3^{k}+\cdots&=-\frac{B_{k+1}(\epsilon)}{k+1},
\end{aligned}
\end{equation*}
where $E_{k}(0)$, $B_{k}$ and $B_{k}(\epsilon)$ are the  Euler polynomials at 0, the Bernoulli numbers and the Apostol--Bernoulli numbers, respectively.
 \end{abstract}

\maketitle

\section{Introduction}\label{Sec.1}
In  the classical analysis, it is well known that
\begin{equation*}
N=1+2+3+\cdots=+\infty.
\end{equation*}
But in the quantum field theory, the Casimir effect indicates an absurd formula
\begin{equation}\label{(1)+}
N=1+2+3+\cdots=-\frac{1}{12}.
\end{equation}
There are several explanations of the above summation, including the Abel summation by using the power series 
(\cite[p. 54]{Stein} or \cite[Section 8.2]{Stopple}) and the analytic continuation of zeta functions (\cite[Section 8.4]{Stopple}).
Recently, Sasane \cite{Amol} gave a new explanation (\ref{(1)+}) based on the
Fourier series of periodic distributions.

His approach is as follows. Denote by the alternating series
\begin{equation}\label{(2)+}
A=1-2+3-4+\cdots.
\end{equation}
As a part of the definition of the summation method, by using a homothetic transformation which doubled the period of a periodic distribution (see \cite[Definition 2.8]{Amol}), 
Sasane obtained the following relationship between (\ref{(1)+}) and  (\ref{(2)+}):
\begin{equation}\label{(3)+} 
\begin{aligned}
A&=(1+2+3+4+\cdots)-2(2+4+6+8+\cdots)\\
&=(1+2+3+4+\cdots)-2^{2}(1+2+3+4+\cdots)\\
&=(1-2^{2})N=-3N
\end{aligned}
 \end{equation}
 (see \cite[p. 478, Eq. (2)]{Amol}).
Then he investigated the following $2\pi$-periodic distribution
$$D=e^{it}-2e^{2it}+3e^{3it}-4e^{4it}+\cdots\in\calD'(\mR),$$
which corresponds to (\ref{(2)+}), and showed that
\begin{equation*}
\begin{aligned}
D={\mathrm P\mathrm f} {\displaystyle \frac{e^{it}}{(1+e^{it})^2}}
+i\pi \displaystyle \sum_{n\in \mZ} \delta'_{(2n+1)\pi},
\end{aligned}
\end{equation*}
where
${\mathrm P\mathrm f} {\displaystyle \frac{e^{it}}{(1+e^{it})^2}}\in \calD^{\prime}(\mR)$
is the $2\pi$-periodic distribution given by
 \begin{equation}\label{(4)+}
\left\langle {\mathrm P\mathrm f} {\displaystyle \frac{e^{it}}{(1\!+\!e^{it})^2}} ,\varphi \right\rangle
=\lim_{\epsilon\searrow 0}
 \left( \int_{(-\delta,\pi-\epsilon)\cup(\pi+\epsilon,2\pi+\delta)}\! \frac{\varphi(t) e^{it}}{(1\!+\!e^{it})^2}dt \!-\!\frac{\varphi(\pi)}{\tan (\epsilon/2)}\right)
\end{equation}
for $ \varphi \in \calD(\mR)$ with support $\textrm{supp}(\varphi)\subset (-\delta,2\pi+\delta)$, where $\delta\in (0,\pi)$.

Sasane introduced the following generalized summation method
by ``evaluating a distribution at a point".

\begin{definition}[{\cite[Definition 2.4]{Amol}}]\label{Def 2.4}
\noindent
For a $2\pi$-periodic distribution $T$, let
$$
\displaystyle
\sum_{n\in \mZ} c_n(T)e^{int}=T
$$
in $\calD'(\mR)$.

\medskip

If there exists a $\sigma\in \mC$ such that for any
approximate identity $\{\varphi_m\}_{m=1}^{\infty}$, the limit
$\displaystyle \lim_{m\rightarrow \infty} \langle T, \varphi_m\rangle$
exists, and
 $\displaystyle
\lim_{m\rightarrow \infty} \langle T, \varphi_m\rangle=\sigma,
$
then we say  the series
$$
\displaystyle \sum_{n\in \mZ} c_n(T)
$$
  is summable, or the   Fourier series of $T$ is summable at
  $t=0$, and we define
$$
\sum_{n\in \mZ} c_n(T)=\sigma.
$$
\end{definition}

\begin{remark}
Recently, Estrada and Kellinsky-Gonzales \cite{EK} characterized Sasane's method in terms of the classical notion of distributional point values (essentially) introduced by Lojasiewicz long ago. This together with characterizations of  Estrada and  Vindas of (symmetric) distributional point values appears to lead to an important connection of the summability method of Sasane and Cesaro summability. For an extended explanation of the connection between Cesaro summability and periodic distributions (i.e., the so-called there Hardy-Littlewood problem for symmetric summability), see the last chapter of the book \cite{PSV}.
\end{remark}

Now let $\{\varphi_m\}_{m=1}^{\infty}$ be any approximate identity,
that is, $$\lim_{m\to\infty}\varphi_m=\delta_{0}$$
in $\calD'(\mR)$.  From (\ref{(4)+}), Sasane \cite{Amol}  proved
\begin{equation*}
\lim_{m\to\infty} \left\langle D,\varphi_m\right\rangle=\frac{1}{4},
\end{equation*}
which leads to the summability of (\ref{(2)+}) in the sense of Definition \ref{Def 2.4} and the sum is
$A=\frac{1}{4}.$ Then from (\ref{(3)+}) we have $$N=-\frac{1}{3}A=-\frac{1}{12},$$
which is the same as (\ref{(1)+}) predicated by the Casimir effect.

In this paper, we use the same summation method  from \cite{Amol} to treat a class of examples. That is, we consider the summability of the generalized series $$a_{1}+2a_{2}+3a_{3}+\cdots+na_{n}+\cdots$$ in the sense of distributions
(see Definition \ref{Def 2.4} above), where $\{a_{n}\}_{n=1}^{\infty}$ is a sequence
that grows at
most polynomially and satisfies an extra assumption which will be introduced below. Here the term ``grow at
most polynomially" means that,  for some $M>0$ and $k>0$ we have
\begin{equation}\label{(1)}
\textrm{ for all }n\in \mN,\,\,\, |a_n|\leq M(1+n)^k.
\end{equation}
Under the above condition, the Fourier series
$$D=a_{1}e^{it}+2a_{2}e^{2it}+3a_{3}e^{3it}+\cdots+na_{n}e^{nit}+\cdots$$
converges in the sense of distributions,
that is, $D\in\calD'(\mR)$. Let
\begin{equation}\label{in+A} f(z)=a_{1}z+a_{2}z^{2}+a_{3}z^{3}+\cdots+a_{n}z^{n}+\cdots\end{equation}
be the corresponding power series. From the Cauchy-Hadamard
theorem, it will converge in the domain $$D_{R}=\{z\in\mathbb{C}~|~|z|<R\},$$
where \begin{equation}\label{(2)}
\frac{1}{R}=\varlimsup_{n\to\infty}\sqrt[n]{|a_{n}|}.
\end{equation}
By the polynomial growth condition (\ref{(1)}),
we have
\begin{equation*}
\frac{1}{R}=\varlimsup_{n\to\infty}\sqrt[n]{|a_{n}|}\leq\lim_{n\to\infty}\sqrt[n]{M(1+n^k)}=1
\end{equation*}
and $R\geq 1$. Let $g(t)=f(e^{it})$ be the corresponding $2\pi$-periodic function on $\mathbb{R}$. Then we have $g^{\prime}(t)=if^{\prime}(e^{it})e^{it}.$

If $R>1$, then the power series $f(z)$ is analytic on the closed disc $\overline{D_{1}}=\{z\in\mathbb{C}~|~|z|\leq 1\}$.
We will show that  $$D=a_{1}e^{it}+2a_{2}e^{2it}+3a_{3}e^{3it}+\cdots+na_{n}e^{nit}+\cdots={\mathrm P\mathrm f} {\displaystyle\left(f^{\prime}(e^{it})e^{it}\right)},$$
 where ${\mathrm P\mathrm f} {\displaystyle\left(f^{\prime}(e^{it})e^{it}\right)}\in\calD'(\mR)$ is the $2\pi$-periodic distribution given by
\begin{equation}\label{in+1}
\left\langle {\mathrm P\mathrm f} {\displaystyle\left(f^{\prime}(e^{it})e^{it}\right)} ,\varphi \right\rangle
=
\int_{0}^{2\pi}f^{\prime}(e^{it})e^{it}\varphi(t)dt
\end{equation}
for $ \varphi \in \calD((0,2\pi))$ (see Theorem \ref{2.2} below).

If $R=1$, then $f(z)$ is analytic in the open disc  $D_{1}=\{z\in\mathbb{C}~|~|z|<1\}$.
We suppose that it can be analytic continued to some larger area in the complex plane contained $D_{1}$.
In this case, there may have some singular points on the unit circle
$C_{1}=\{z\in\mathbb{C}~|~|z|=1\}$. Here we assume that there exists only one pole $z_{0}\neq 1$
with order 1 on $C_{1}$ and $f(z)$ has the following Laurent series expansion
\begin{equation}\label{(*)}
f(z)=\frac{c_{-1}}{z-z_{0}}+c_{0}+c_{1}(z-z_{0})+c_{2}(z-z_{0})^{2}+\cdots
\end{equation}
in the annular $$D_{z_{0}}=\{z\in\mathbb{C}~|~0< |z-z_{0}|<r\}$$ and $2< r<+\infty$.
Denote by $$g(z)=c_{0}+c_{1}(z-z_{0})+c_{2}(z-z_{0})^{2}+\cdots$$
its analytic part. Then we have
$$f(z)=\frac{c_{-1}}{z-z_{0}}+g(z).$$
So $$f^{\prime}(z)=\frac{-c_{-1}}{(z-z_{0})^{2}}+g^{\prime}(z).$$
For simplification of the notations, let $d_{-1}=-c_{-1}$, we have
\begin{equation}\label{in+3} f^{\prime}(z)=\frac{d_{-1}}{(z-z_{0})^{2}}+g^{\prime}(z).\end{equation}
Denote by $z_{0}=e^{it_{0}}$ for some $t_{0}\neq 0$ (since $z_{0}\neq 1$ by our assumption).
In this case, we show that  \begin{equation}\label{defD}
\begin{aligned}
D&=a_{1}e^{it}+2a_{2}e^{2it}+3a_{3}e^{3it}+\cdots+na_{n}e^{nit}+\cdots
\\
&={\mathrm P\mathrm f} {\displaystyle\left(f^{\prime}(e^{it})e^{it}\right)}-d_{-1}i\pi e^{-it_{0}} \displaystyle \sum_{n\in \mZ} \delta'_{t_{0}+2n\pi},
\end{aligned}
\end{equation}
where ${\mathrm P\mathrm f} {\displaystyle\left(f^{\prime}(e^{it})e^{it}\right)}\in\calD'(\mR)$ is the $2\pi$-periodic distribution given by
\begin{equation}\label{in+2}
\begin{aligned}
&\quad\left\langle {\mathrm P\mathrm f} {\displaystyle\left(f^{\prime}(e^{it})e^{it}\right)} ,\varphi \right\rangle\\
&=
 \lim_{\epsilon\searrow 0}
 \left( \int_{(t_{0}-\pi,t_{0}-\epsilon)\cup(t_{0}+\epsilon,t_{0}+\pi)}f^{\prime}(e^{it})e^{it}\varphi(t)dt+\frac{d_{-1}e^{-it_{0}}}{\tan (\epsilon/2)}\varphi(t_{0})\right)
\end{aligned}
\end{equation}
for $ \varphi \in \calD(\mR)$ with support $\textrm{supp}(\varphi)\subset (t_{0}-\pi-\delta,t_{0}+\pi+\delta)$, where $\delta\in (0,\pi)$ (see Theorem \ref{3.2} below).

In both cases ($R>1$ and $R=1$), we will show the summation $$a_{1}+2a_{2}+3a_{3}+\cdots+na_{n}+\cdots=f^{\prime}(1)$$ in the sense of distributions.
More generally, for any $k\in\mathbb{N},$  we have
$$1^{k}a_{1}+2^{k}a_{2}+3^{k}a_{3}+\cdots+n^{k}a_{n}+\cdots=\frac{1}{i^{k}}\left(\frac{d}{dt}\right)^{k}f(e^{it})\bigg|_{t=0}$$
(see Theorem \ref{4.1} below).

As applications, in the last section, we shall explain  the following summation formulas
\begin{equation}\label{(1.13)}
\begin{aligned}
1^{k}-2^{k}+3^{k}-\cdots&=-\frac{E_{k}(0)}{2},
\\
1^{k}+2^{k}+3^{k}+\cdots&=-\frac{B_{k+1}}{k+1},
\\
\epsilon^{1}1^{k}+\epsilon^{2}2^{k}+\epsilon^{3}3^{k}+\cdots&=-\frac{B_{k+1}(\epsilon)}{k+1},
\end{aligned}
\end{equation}
where $E_{k}(0)$, $B_{k}$ and $B_{k}(\epsilon)$ are the  Euler polynomials at 0, the Bernoulli numbers and the Apostol--Bernoulli numbers, respectively.

\begin{remark} Recently, in an email, Christian Schubert told us that they have developed a formalism in the quantum field theory context that seems related to this work. If one takes (7.1) in \cite{Schubert1} with even $n$ and then evaluates the left-hand side in the Fourier basis (excluding the constant functions), one obtains the middle equation in (\ref{(1.13)}). Similarly, the first equation in (\ref{(1.13)}) could be obtained by switching from periodic to anti-periodic boundary conditions in the Hilbert space, see, e.g., equation (F.2) of \cite{Schubert2}. The third equation in (\ref{(1.13)}) could presumably be obtained using twisted boundary conditions.
\end{remark}

\section{$R>1$}
In this case, we first prove the following proposition which shows that ${\mathrm P\mathrm f} {\displaystyle\left(f^{\prime}(e^{it})e^{it}\right)}$
(see (\ref{in+1})) defines a distribution on the open interval $(0,2\pi)$.
\begin{proposition}\label{Proposition 2.1}
For $ \varphi \in \calD((0,2\pi)),$ define
\begin{eqnarray*}
\left\langle {\mathrm P\mathrm f} {\displaystyle\left(f^{\prime}(e^{it})e^{it}\right)} ,\varphi \right\rangle
\!\!\!\!\!&=&\!\!\!\!\!
\int_{0}^{2\pi}f^{\prime}(e^{it})e^{it}\varphi(t)dt.
\end{eqnarray*}
 Then ${\mathrm P\mathrm f} {\displaystyle\left(f^{\prime}(e^{it})e^{it}\right)}\in\calD'((0,2\pi))$.
 \end{proposition}
 \begin{proof}
 It only needs to prove the continuity. Since $f(z)$ is analytic, $f^{\prime}(e^{it})$ is continuous and bounded on $\mathbb{R}$, that is,
 there exists a constant $m>0$, such that
\begin{equation*}
|f^{\prime}(e^{it})|\leq m\quad\textrm{for}~~t\in [0,2\pi].
\end{equation*}
Let $\{\varphi_n\}_{n=1}^{\infty}$ be a sequence converges  to 0 in $\calD((0,2\pi))$, then in particular,
$\{\varphi_n\}_{n=1}^{\infty}$ converges to 0 uniformly on $(0,2\pi)$. Thus
\begin{equation*}
\begin{aligned}
\left|\left\langle {\mathrm P\mathrm f} {\displaystyle\left(f^{\prime}(e^{it})e^{it}\right)} ,\varphi_{n} \right\rangle\right|
&=\left|\int_{0}^{2\pi}f^{\prime}(e^{it})e^{it}\varphi_{n}(t)dt\right|\\
&\leq \int_{0}^{2\pi}|f^{\prime}(e^{it})||\varphi_{n}(t)|dt\\
&\leq2\pi m\|\varphi_{n}\|_{\infty}
\end{aligned}
\end{equation*}
 and
 $$\lim_{n\to\infty}\left|\left\langle {\mathrm P\mathrm f} {\displaystyle\left(f^{\prime}(e^{it})e^{it}\right)} ,\varphi_{n} \right\rangle\right|\leq\lim_{n\to\infty}2\pi m\|\varphi_{n}\|_{\infty}=0.$$
 So ${\mathrm P\mathrm f} {\displaystyle\left(f^{\prime}(e^{it})e^{it}\right)}\in\calD'((0,2\pi))$.
  \end{proof}
  Now we calculate the Fourier coefficients of the distribution $D={\mathrm P\mathrm f} {\displaystyle\left(f^{\prime}(e^{it})e^{it}\right)}$
  and show that
  \begin{theorem}\label{2.2}
  The following Fourier series expansion is valid in $\calD'(\mathbb{R}):$
  $$D={\mathrm P\mathrm f} {\displaystyle\left(f^{\prime}(e^{it})e^{it}\right)}=a_{1}e^{it}+2a_{2}e^{2it}+3a_{3}e^{3it}+\cdots+na_{n}e^{nit}+\cdots.$$
  \end{theorem}
  \begin{proof}
  For each $\delta>0$, let $\rho_{\delta}\in\calD'(\mathbb{R})$ be the test function constructed
  as in \cite[p. 489]{Amol} such that  $\rho_{\delta}\big|_{(\delta,2\pi-\delta)}=1$.
  Define $\varphi_{\delta}(t)\in\calD(\mathbb{R})$ by
  $$\varphi_{\delta}(t)=\rho_{\delta}(t)e^{-nit}.$$
  So as $\delta\searrow 0,$
  $\varphi_{\delta}$ converges pointwise to 1 on $(0,2\pi)$, and to 0 on $\mathbb{R}\setminus [0,2\pi].$
 Then we have
$$
\varphi_{\delta,{ \scriptscriptstyle\textrm{circle}}}(e^{it})
=\sum_{m\in \mZ} \varphi_\delta(t-2\pi m)
=\sum_{m\in \mZ} \rho_\delta(t-2\pi m) e^{-nit}
=e^{-nit}.
$$
Hence
$$
c_n(D)
=\frac{1}{2\pi}\langle D_{\scriptscriptstyle\textrm{circle}},e^{-nit}\rangle
=
\frac{1}{2\pi}\langle D_{ \scriptscriptstyle\textrm{circle}} ,
\varphi_{\delta,{\scriptscriptstyle\textrm{circle}}} \rangle=
\frac{1}{2\pi}\langle D,\varphi_{\delta} \rangle
$$
 (see \cite[p. 490]{Amol}).
 Consider the interval $O=(-\frac{\pi}{2},\frac{3\pi}{2}).$
 Since
  \begin{equation*}
\begin{aligned}
 \langle D,\varphi_{\delta}\rangle &=\left\langle {\mathrm P\mathrm f} {\displaystyle\left(f^{\prime}(e^{it})e^{it}\right)} ,\varphi_{\delta} \right\rangle
 \\&=\lim_{\delta\searrow 0} \int_{O}f^{\prime}(e^{it})e^{it}\varphi_{\delta}(t)dt\\
&=\lim_{\delta\searrow 0} \int_{O}f^{\prime}(e^{it})e^{it}\rho_{\delta}(t)e^{-nit}dt\\
&=\int_{0}^{2\pi}f^{\prime}(e^{it})e^{it}e^{-nit}dt,
\end{aligned}
\end{equation*}
the $n$th Fourier coefficient of the distribution $D$ is given by
\begin{equation*}
\begin{aligned}
c_n(D)&=
\frac{1}{2\pi}\langle D,\varphi_{\delta} \rangle\\
&=
\frac{1}{2\pi}\int_{0}^{2\pi}f^{\prime}(e^{it})e^{it}e^{-nit}dt\\
&=\frac{1}{2\pi i}\oint_{C_{1}}f^{\prime}(z)z^{-n}dz,
\end{aligned}
\end{equation*}
where $C_{1}$ is the unit circle $\{z\in\mathbb{C}~|~|z|=1\}.$
Because $$f^{\prime}(z)=a_{1}+2a_{2}z+3a_{3}z^{2}+\cdots+na_{n}z^{n-1}+\cdots,$$
 we have
\begin{equation*}
c_n(D)=\frac{1}{2\pi i}\oint_{C_{1}}f^{\prime}(z)z^{-n}dz=\textrm{Res}_{z=0}\left(f^{\prime}(z)z^{-n}\right)=na_{n},
\end{equation*}
as desired.   \end{proof}

  \section{$R$=1}
In this case, as the previous discussion, first we also  show that ${\mathrm P\mathrm f} {\displaystyle\left(f^{\prime}(e^{it})e^{it}\right)}$
(see (\ref{in+2})) defines a distribution on the open interval $(t_{0}-\pi,t_{0}+\pi)$.
\begin{proposition}
For any $\varphi\in\calD((t_{0}-\pi,t_{0}+\pi)),$ we have
\begin{equation}
\begin{aligned}
\left\langle {\mathrm P\mathrm f} {\displaystyle\left(f^{\prime}(e^{it})e^{it}\right)} ,\varphi \right\rangle
&=d_{-1}\int_{t_{0}-\pi}^{t_{0}+\pi}\frac{(t-t_{0})^{2}}{(e^{it}-e^{it_{0}})^{2}}e^{it}\int_{0}^{1}(1-\theta)\varphi^{\prime\prime}(t_{0}+\theta(t-t_{0}))d\theta dt\\
&\quad+\int_{t_{0}-\pi}^{t_{0}+\pi}g^{\prime}(e^{it})e^{it}\varphi(t)dt,
\end{aligned}
\end{equation}
where some $\theta\in(0,1)$ and
${\mathrm P\mathrm f} {\displaystyle\left(f^{\prime}(e^{it})e^{it}\right)}\in\calD^{\prime}((t_{0}-\pi,t_{0}+\pi)).$
\end{proposition}
\begin{proof}
We extend the argument of \cite[Proposition 3.1]{Amol} to our case.
By (\ref{in+2}) and  (\ref{in+3}), we have
\begin{equation}\label{2}
\begin{aligned}
&\quad\left\langle {\mathrm P\mathrm f} {\displaystyle\left(f^{\prime}(e^{it})e^{it}\right)} ,\varphi \right\rangle\\
&=
 \lim_{\epsilon\searrow 0}
 \left( \int_{\Omega_{\epsilon}}f^{\prime}(e^{it})e^{it}\varphi(t)dt+\frac{d_{-1}e^{-it_{0}}}{\tan (\epsilon/2)}\varphi(t_{0})\right)\\
 &=d_{-1}\underbrace{\lim_{\epsilon\searrow 0}\int_{\Omega_{\epsilon}} \frac{e^{it}}{(e^{it}-e^{it_{0}})^{2}}\varphi(t)dt}_{(\textrm{I})}+\underbrace{\lim_{\epsilon\searrow 0}\int_{\Omega_{\epsilon}}g^{\prime}(e^{it})e^{it}\varphi(t)dt}_{(\textrm{II})}\\
 &\quad+\lim_{\epsilon\searrow 0}\frac{d_{-1}e^{-it_{0}}}{\tan (\epsilon/2)}\varphi(t_{0}),
   \end{aligned}
\end{equation}
where $\Omega_{\epsilon}=(t_{0}-\pi,t_{0}-\epsilon)\cup(t_{0}+\epsilon,t_{0}+\pi).$

First we calculate (I). By Taylor's formula, we have
$$
\varphi(t)= \varphi(t_0)+(t-t_{0})\varphi'(t_0)
           +(t-t_0)^2 \int_0^1 (1-\theta) \varphi''(t_0+\theta(t-t_0)) d\theta
$$
for some $\theta\in(0,1)$.
Thus \begin{equation}\label{star}
\begin{aligned}
\int_{\Omega_{\epsilon}} \frac{e^{it}}{(e^{it}-e^{it_{0}})^{2}}\varphi(t)dt&=\varphi(t_{0})\underbrace{\int_{\Omega_{\epsilon}} \frac{e^{it}}{(e^{it}-e^{it_{0}})^{2}}dt}_{\textrm{A}}\\
&\quad+\varphi^{\prime}(t_{0})\underbrace{\int_{\Omega_{\epsilon}} \frac{(t-t_{0})e^{it}}{(e^{it}-e^{it_{0}})^{2}}dt}_{\textrm{B}}\\
&\quad+\underbrace{\int_{\Omega_{\epsilon}}\frac{(t-t_{0})^{2}}{(e^{it}-e^{it_{0}})^{2}}e^{it}\int_{0}^{1}(1-\theta)\varphi^{\prime\prime}(t_{0}+\theta(t-t_{0}))d\theta dt}_{\textrm{C}}.
\end{aligned}
\end{equation}
To compute A, note that
$$\frac{d}{dt}\frac{1}{e^{it}-e^{it_{0}}}=-\frac{ie^{it}}{(e^{it}-e^{it_{0}})^{2}},$$
we have
\begin{equation}\label{(6)}
\begin{aligned}
A&=\int_{\Omega_{\epsilon}} \frac{e^{it}}{(e^{it}-e^{it_{0}})^{2}}dt\\
&=\frac{1}{i}\int_{\Omega_{\epsilon}} \frac{1}{(e^{it}-e^{it_{0}})^{2}}d(e^{it}-e^{it_0})\\
&=-\frac{1}{i}\left(\frac{1}{e^{it}-e^{it_{0}}}\bigg|_{t_{0}-\pi}^{t_{0}-\epsilon}+\frac{1}{e^{it}-e^{it_{0}}}\bigg|_{t_{0}+\epsilon}^{t_{0}+\pi}\right)\\
&=i\bigg(\frac{1}{e^{it_{0}}e^{-i\epsilon}-e^{it_{0}}}-\frac{1}{e^{it_{0}}e^{-i\pi}-e^{it_{0}}}+\frac{1}{e^{it_{0}}e^{i\pi}-e^{it_{0}}}\\
&\quad-\frac{1}{e^{it_{0}}e^{i\epsilon}-e^{it_{0}}}\bigg)\\
&=\frac{i}{e^{it_0}}\left(\frac{1}{e^{-i\epsilon}-1}-\frac{1}{e^{i\epsilon}-1}\right)\\
&=-\frac{i}{e^{it_0}}\frac{e^{i\epsilon}+1}{e^{i\epsilon}-1}\\
&=-\frac{i}{e^{it_0}}\frac{e^{i\frac{\epsilon}{2}}+e^{-i\frac{\epsilon}{2}}}{e^{i\frac{\epsilon}{2}}-e^{-i\frac{\epsilon}{2}}}\\
&=-\frac{1}{e^{it_0}}\frac{\cos(\epsilon/2)}{\sin(\epsilon/2)}\\
&=-\frac{1}{e^{it_0}}\frac{1}{\tan(\epsilon/2)}.
\end{aligned}
\end{equation}
For B, we have
\begin{equation}
\begin{aligned}
B&=\int_{\Omega_{\epsilon}} \frac{(t-t_0)e^{it}}{(e^{it}-e^{it_{0}})^{2}}dt\\
&=\int_{t_{0}-\pi}^{t_{0}-\epsilon} \frac{(t-t_0)e^{it}}{(e^{it}-e^{it_{0}})^{2}}dt+\int_{t_{0}+\epsilon}^{t_{0}+\pi} \frac{(t-t_0)e^{it}}{(e^{it}-e^{it_{0}})^{2}}dt\\
&=\int_{t_{0}-\pi}^{t_{0}-\epsilon} \frac{(t-t_0)e^{it}}{(e^{it}-e^{it_{0}})^{2}}dt+\int_{t_{0}+\epsilon}^{t_{0}+\pi} \frac{(t_0-\tau)e^{2it_0}\cdot e^{-i\tau}}{(e^{2it_0}\cdot e^{-i\tau}-e^{it_{0}})^{2}}(-1)d\tau\\
 &\quad (\textrm{letting $\tau=2t_{0}-t$})\\
&=\int_{t_{0}-\pi}^{t_{0}-\epsilon} \frac{(t-t_0)e^{it}}{(e^{it}-e^{it_{0}})^{2}}dt+\int_{t_{0}-\epsilon}^{t_{0}-\pi}\frac{(\tau-t_{0})e^{i\tau}}{(e^{i\tau}-e^{it_{0}})^{2}}d\tau\\
&=0.
\end{aligned}
\end{equation}
Then we consider the integral C.
By Taylor's formula
$$e^{it}=e^{it_{0}}+i(t-t_{0})e^{it_{0}}+O((t-t_{0})^{2}),$$
we have $$\lim_{t\to t_{0}}\left|\frac{e^{it}-e^{it_{0}}}{t-t_0}\right|=1.$$
So by the continuity, there exists a constant $m_1>0$ such that
$$\left|\frac{e^{it}-e^{it_{0}}}{t-t_0}\right|\geq m_{1}$$
for $t\in [t_{0}-\pi,t_{0}+\pi]$, thus
\begin{equation*}
\begin{aligned}
&\quad\lim_{\epsilon\searrow 0}\int_{\Omega_{\epsilon}}\frac{(t-t_{0})^{2}}{(e^{it}-e^{it_{0}})^{2}}e^{it}\int_{0}^{1}(1-\theta)\varphi^{\prime\prime}(t_{0}+\theta(t-t_{0}))d\theta dt\\
&=\int_{t_0-\pi}^{t_0+\pi}\frac{(t-t_{0})^{2}}{(e^{it}-e^{it_{0}})^{2}}e^{it}\int_{0}^{1}(1-\theta)\varphi^{\prime\prime}(t_{0}+\theta(t-t_{0}))d\theta dt.
\end{aligned}
\end{equation*}
Substituting the above computations of A, B and C  into (\ref{star}),
we have
\begin{equation}\label{bian1}
\begin{aligned}
\textrm{(I)}&=\lim_{\epsilon\searrow 0}\int_{\Omega_{\epsilon}} \frac{e^{it}}{(e^{it}-e^{it_{0}})^{2}}\varphi(t)dt\\
&=\int_{t_0-\pi}^{t_0+\pi}\frac{(t-t_{0})^{2}}{(e^{it}-e^{it_{0}})^{2}}e^{it}\int_{0}^{1}(1-\theta)\varphi^{\prime\prime}(t_{0}+\theta(t-t_{0}))d\theta dt\\&\quad-\lim_{\epsilon\searrow 0}\frac{1}{e^{it_0}}\frac{\varphi(t_{0})}{\tan(\epsilon/2)}.
\end{aligned}
\end{equation}
For (II), since $g^{\prime}(e^{it})$ is continuous on the interval $[t_{0}-\pi,t_{0}+\pi]$, we have
\begin{equation}\label{bian2}
\textrm{(II)}=\lim_{\epsilon\searrow 0}\int_{\Omega_{\epsilon}}g^{\prime}(e^{it})e^{it}\varphi(t)dt=\int_{t_{0}-\pi}^{t_{0}+\pi}g^{\prime}(e^{it})e^{it}\varphi(t)dt.
\end{equation}
Then subsititute (\ref{bian1}) and (\ref{bian2}) into  (\ref{2}) we get
\begin{equation*}
\begin{aligned}
\left\langle {\mathrm P\mathrm f} {\displaystyle\left(f^{\prime}(e^{it})e^{it}\right)} ,\varphi \right\rangle
&=d_{-1}\int_{t_{0}-\pi}^{t_{0}+\pi}\frac{(t-t_{0})^{2}}{(e^{it}-e^{it_{0}})^{2}}e^{it}\int_{0}^{1}(1-\theta)\varphi^{\prime\prime}(t_{0}+\theta(t-t_{0}))d\theta dt\\
&\quad+\int_{t_{0}-\pi}^{t_{0}+\pi}g^{\prime}(e^{it})e^{it}\varphi(t)dt.
\end{aligned}
\end{equation*}
To show  the continuity of  $\left\langle {\mathrm P\mathrm f} {\displaystyle\left(f^{\prime}(e^{it})e^{it}\right)} ,\varphi \right\rangle$,
let $\{\varphi_{n}\}_{n=1}^{\infty}$ be a sequence converges to 0 in the topology of $\calD((t_{0}-\pi,t_{0}+\pi)).$
Then, in particular, $\{\varphi_{n}\}_{n=1}^{\infty}$ and $\{\varphi_{n}^{\prime\prime}\}_{n=1}^{\infty}$ converge uniformly to 0
on  $(t_{0}-\pi,t_{0}+\pi)$. Since $g^{\prime}(e^{it})$ is continuous on the closed interval  $[t_{0}-\pi,t_{0}+\pi]$, there exists a constant
$m_2\geq 0,$
such that
\begin{equation*}
|g^{\prime}(e^{it})|\leq m_{2}\quad\textrm{for}~~t\in [t_0-\pi,t_{0}+\pi].
\end{equation*}
So
\begin{equation*}
\begin{aligned}
\left|\left\langle {\mathrm P\mathrm f} {\displaystyle\left(f^{\prime}(e^{it})e^{it}\right)} ,\varphi_{n}\right\rangle\right|
&\leq |d_{-1}|\left|\int_{t_{0}-\pi}^{t_{0}+\pi}\frac{(t-t_{0})^{2}}{(e^{it}-e^{it_{0}})^{2}}e^{it}\int_{0}^{1}(1-\theta)\varphi_{n}^{\prime\prime}(t_{0}+\theta(t-t_{0}))d\theta dt\right|\\
&\quad+\left|\int_{t_{0}-\pi}^{t_{0}+\pi}g^{\prime}(e^{it})e^{it}\varphi_{n}(t)dt\right|\\
&\leq |d_{-1}|\cdot 2\pi\cdot\frac{1}{m_{1}^{2}}\cdot||\varphi^{\prime\prime}_{n}||_{\infty}+2\pi\cdot m_{2}\cdot||\varphi_{n}||_{\infty}
\end{aligned}
\end{equation*}
and $$\lim_{n\to\infty}\left\langle {\mathrm P\mathrm f} {\displaystyle\left(f^{\prime}(e^{it})e^{it}\right)} ,\varphi_{n} \right\rangle=0.$$
In conclusion, ${\mathrm P\mathrm f} {\displaystyle\left(f^{\prime}(e^{it})e^{it}\right)}\in\calD^{\prime}((t_{0}-\pi,t_{0}+\pi)).$ \end{proof}
Now we calculate the Fourier coefficients of the distribution $$D={\mathrm P\mathrm f} {\displaystyle\left(f^{\prime}(e^{it})e^{it}\right)}-d_{-1}i\pi e^{-it_{0}} \displaystyle \sum_{n\in \mZ} \delta'_{t_{0}+2n\pi},$$ 
and we have the following result.

\begin{theorem}\label{3.2}
The following Fourier series expansion is valid in $\calD^{\prime}(\mathbb{R}):$
 $$D=a_{1}e^{it}+2a_{2}e^{2it}+3a_{3}e^{3it}+\cdots+na_{n}e^{nit}+\cdots.$$\end{theorem}
 \begin{proof}
 We extend the argument of \cite[Theorem 5.1]{Amol} to our case.
 For each $\delta>0$, let
  $\rho_\delta \in \calD(\mR)$ be the test function such that
  $$
\rho_{\delta}\Big|_{(t_0-\pi+\delta,t_{0}+\pi-\delta)}=1.
$$
Thus
$$
\sum_{n\in \mZ} \rho_\delta (t+2\pi n)=1\quad (t\in \mR).
$$

Similar with the arguments in \cite[p. 489]{Amol}, here $\rho_\delta$ can be constructed as follows. 
For $\delta>0$ being small, consider any symmetric,
nonnegative test function $\varphi$ with support in $[t_{0}-\pi-\delta,t_{0}-\pi+\delta]$ 
and such that
$$
\int_{t_{0}-\pi-\delta}^{t_{0}-\pi+\delta} \varphi(t)dt=1.
$$
Define the function
$$
\Phi(t):=\int_{(-\infty,t]} \varphi(\tau) d\tau .
$$
Then it can be seen that for all $t\in \mR$, $\Phi$ satisfies
$\Phi(t_0-\pi+t)+\Phi(t_0-\pi-t)=1$.
So $\rho_\delta$ can be defined by
$$
\rho_\delta (t)=\Phi(t)\cdot \Phi(2t_{0}-t).
$$
Then we define $\varphi_\delta \in \calD(\mR)$ by
$$
\varphi_\delta (t)=\rho_\delta(t) e^{-nit}
$$
and
$$
\varphi_{\delta,{ \scriptscriptstyle\textrm{circle}}}(e^{it})
=\sum_{m\in \mZ} \varphi_\delta(t-2\pi m)
=\sum_{m\in \mZ} \rho_\delta(t-2\pi m) e^{-nit}
=e^{-nit}.
$$
Hence the $n$th Fourier coefficients of $D$ can be calculated by
$$
c_n(D)
=\frac{1}{2\pi}\langle D_{\scriptscriptstyle\textrm{circle}},e^{-nit}\rangle
=
\frac{1}{2\pi}\langle D_{ \scriptscriptstyle\textrm{circle}} ,
\varphi_{\delta,{\scriptscriptstyle\textrm{circle}}} \rangle=
\frac{1}{2\pi}\langle D,\varphi_{\delta} \rangle
$$
(see \cite[p. 490]{Amol}).

Let $$O_{\epsilon}=\left(t_{0}-\frac{3}{2}\pi,t_0-\epsilon\right)\cup\left(t_0+\epsilon,t_0+\frac{3}{2}\pi\right)$$
and we partition it as
\begin{equation*}
\begin{aligned}
O_{\epsilon}&=\underbrace{\left(t_{0}-\frac{3}{2}\pi,t_0-\frac{\pi}{2}\right]\cup\left[t_0+\frac{\pi}{2}, t_{0}+\frac{3}{2}\pi\right)}_{V}\\
&\quad\cup\underbrace{\left(t_{0}-\frac{\pi}{2},t_0-\epsilon\right)\cup\left(t_0+\epsilon,t_0+\frac{\pi}{2}\right)}_{V_\epsilon}.
\end{aligned}
\end{equation*}
Then by (\ref{in+2}) we have
\begin{equation*}
\begin{aligned}
\left\langle {\mathrm P\mathrm f} {\displaystyle\left(f^{\prime}(e^{it})e^{it}\right)} ,\varphi_{\delta} \right\rangle
&=\lim_{\epsilon\searrow 0} \left(\int_{O_\epsilon}f^{\prime}(e^{it})e^{it}\varphi_{\delta}(t)dt+\frac{d_{-1}\varphi_{\delta}(t_{0})}{e^{it_0}\tan(\epsilon/2)}\right)\\
&=\lim_{\delta\searrow 0} \int_{V}f^{\prime}(e^{it})e^{it}\rho_{\delta}(t)e^{-nit}dt\\
&\quad+\lim_{\epsilon\searrow 0} \left(\int_{V_\epsilon}f^{\prime}(e^{it})e^{it}e^{-nit}dt+\frac{d_{-1}e^{-nit_0}}{e^{it_0}\tan(\epsilon/2)}\right)\\
&= \int_{(t_0-\pi,t_0-\frac{\pi}{2}]\cup[t_0+\frac{\pi}{2},t_0+\pi)}f^{\prime}(e^{it})e^{it}e^{-nit}dt\\
&\quad+\lim_{\epsilon\searrow 0} \left(\int_{V_\epsilon}f^{\prime}(e^{it})e^{it}e^{-nit}dt+\frac{d_{-1}e^{-nit_0}}{e^{it_0}\tan(\epsilon/2)}\right)\\
&=\lim_{\epsilon\searrow 0} \left(\int_{\Omega_\epsilon}f^{\prime}(e^{it})e^{it}e^{-nit}dt+\frac{d_{-1}e^{-nit_0}}{e^{it_0}\tan(\epsilon/2)}\right),
\end{aligned}
\end{equation*}
here we recall that $\Omega_{\epsilon}=(t_{0}-\pi,t_{0}-\epsilon)\cup(t_{0}+\epsilon, t_{0}+\pi).$
Also \begin{equation*}
\left\langle \sum_{n\in \mZ} \delta'_{t_{0}+2n\pi},\varphi_{\delta}\right\rangle=-\varphi_{\delta}^{\prime}(t_0)=-(e^{-nit})^{\prime}\bigg|_{t=t_{0}}=ine^{-nit_0}.
\end{equation*}
Hence
\begin{equation}\label{fivestar}
\begin{aligned}
c_n(D)&=\frac{1}{2\pi}\langle D,\varphi_{\delta} \rangle\\
&=\frac{1}{2\pi}\left\langle{\mathrm P\mathrm f} {\displaystyle\left(f^{\prime}(e^{it})e^{it}\right)}-d_{-1}i\pi e^{-it_{0}} \displaystyle \sum_{n\in \mZ} \delta'_{t_{0}+2n\pi},\varphi_{\delta}\right\rangle\\
&=\frac{1}{2\pi}\Bigg(\lim_{\epsilon\searrow 0}\bigg(\int_{\Omega_\epsilon}f^{\prime}(e^{it})e^{it}e^{-nit}dt+\frac{d_{-1}e^{-nit_0}}{e^{it_0}\tan(\epsilon/2)}\bigg)\\
&\quad+d_{-1}n\pi e^{-(n+1)it_{0}}\Bigg).\end{aligned}
\end{equation}
In what follows, we shall calculate the integral $\frac{1}{2\pi}\int_{\Omega_\epsilon}f^{\prime}(e^{it})e^{it}e^{-nit}dt$.

Let $$C^{c}_{\epsilon}=\left\{z=e^{it}~|~t\in\Omega_{\epsilon}\right\}$$
and $L_{\epsilon}$ be the line connected the points $e^{i(t_0-\epsilon)}$ and  $e^{i(t_0+\epsilon)}.$
Let $$C_{\frac{1}{2}}=\left\{z\in\mathbb{C}~|~|z|=\frac{1}{2}\right\}.$$
Since $f^{\prime}(z)$ is analytic in the open disc $\left\{z\in\mathbb{C}~|~|z|<1\right\},$ by Cauchy's theorem, we have
\begin{equation}\label{(3)}
\begin{aligned}
\frac{1}{2\pi}\int_{\Omega_\epsilon}f^{\prime}(e^{it})e^{it}e^{-nit}dt
&=\frac{1}{2\pi i}\int_{\Omega_\epsilon}f^{\prime}(e^{it})e^{-nit}ie^{it}dt\\
&=\frac{1}{2\pi i}\int_{C^{c}_{\epsilon}}f^{\prime}(z)z^{-n}dz\\
&=\frac{1}{2\pi i}\int_{C^{c}_{\epsilon}}f^{\prime}(z)z^{-n}dz\\
&=\frac{1}{2\pi i}\int_{C^{c}_{\epsilon}+L_{\epsilon}}f^{\prime}(z)z^{-n}dz-\frac{1}{2\pi i}\int_{L_{\epsilon}}f^{\prime}(z)z^{-n}dz\\
&=\frac{1}{2\pi i}\oint_{C_{\frac{1}{2}}}f^{\prime}(z)z^{-n}dz-\frac{1}{2\pi i}\int_{L_{\epsilon}}f^{\prime}(z)z^{-n}dz
\end{aligned}
\end{equation}
(see Fig. 1).

\begin{figure}
\begin{tikzpicture}
circle (center, radius)
\draw [line width=1pt] (0,0) circle (1.5cm);
\draw [line width=1pt] (0,0) circle (3cm);
\draw [line width=1pt] (-3,0)-- (-15/13,36/13);
\draw [line width=1pt,->] (-15/13,36/13)-- (-27/13,18/13);
\draw[line width=1pt, ->] (3,0) arc (0:50:3);
\draw[line width=1pt, ->] (1.5,0) arc (0:50:1.5);
\begin{large}
\draw (-2.05,1.65) node[anchor=north west] {$L_{\epsilon}$};
\draw (2.15,2.75) node[anchor=north west] {$C^{c}_{\epsilon}$};
\draw (1.10,1.60) node[anchor=north west] {$C_{\frac{1}{2}}$};
\end{large}
\end{tikzpicture}
\caption{Path of the integral}
\end{figure}
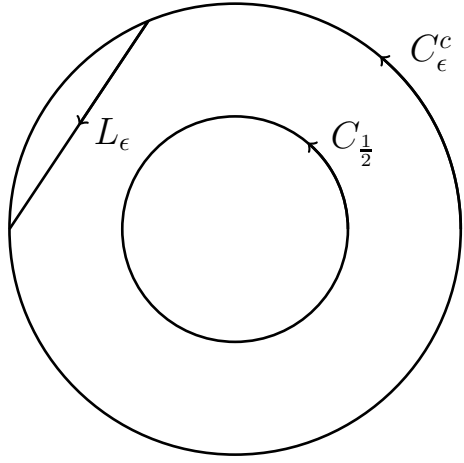

By (\ref{in+A}), we have
$$f^{\prime}(z)=a_{1}+2a_{2}z+\cdots+na_{n}z^{n-1}+\cdots$$
for $|z|<1,$ thus
$$\frac{1}{2\pi i}\oint_{C_{\frac{1}{2}}}f^{\prime}(z)z^{-n}dz=\textrm{Res}_{z=0}\left(f^{\prime}(z)z^{-n}\right)=na_{n}.$$
Substitute into (\ref{(3)}), we have
\begin{equation}\label{(4)}
\frac{1}{2\pi}\int_{\Omega_\epsilon}f^{\prime}(e^{it})e^{it}e^{-nit}dt=na_{n}-\frac{1}{2\pi i}\int_{L_{\epsilon}}f^{\prime}(z)z^{-n}dz.
\end{equation}
Now we calculate the integral $\frac{1}{2\pi i}\int_{L_{\epsilon}}f^{\prime}(z)z^{-n}dz$ in the above equation.
If $n\neq 0$, let $$z^{-n}=z_{0}^{-n}+(-n)z_{0}^{-n-1}(z-z_{0})+O((z-z_0)^{2}))$$
be the Taylor expansion of $z^{-n}$ around $z_{0}$,
then
\begin{equation}\label{(5)}
\begin{aligned}
\frac{1}{2\pi i}\int_{L_{\epsilon}}f^{\prime}(z)z^{-n}dz&=z_{0}^{-n}\frac{1}{2\pi i}\int_{L_{\epsilon}}f^{\prime}(z)dz+(-n)z_{0}^{-n-1}\frac{1}{2\pi i}\int_{L_{\epsilon}}f^{\prime}(z)(z-z_{0})dz\\
&\ll \int_{L_{\epsilon}}f^{\prime}(z)(z-z_{0})^{2}dz.
\end{aligned}
\end{equation}
Since by (\ref{in+3})
 $$f^{\prime}(z)=\frac{d_{-1}}{(z-z_{0})^{2}}+g^{\prime}(z)$$
 for  $z\in D_{z_{0}}=\{z\in\mathbb{C}~|~0< |z-z_{0}|<r\}$ and $2< r<+\infty$,
 $f^{\prime}(z)(z-z_{0})^{2}$ is analytic on the line $L_{\epsilon}$ for any $\epsilon>0$,
 we have \begin{equation*}
 \lim_{\epsilon\searrow 0}\int_{L_{\epsilon}}f^{\prime}(z)(z-z_{0})^{2}dz=0.
 \end{equation*}
So by (\ref{(5)}) we have
 \begin{equation}\label{(7)}
 \begin{aligned}
\lim_{\epsilon\searrow 0}\frac{1}{2\pi i}\int_{L_{\epsilon}}f^{\prime}(z)z^{-n}dz
&=z_{0}^{-n}\lim_{\epsilon\searrow 0}\frac{1}{2\pi i}\int_{L_{\epsilon}}f^{\prime}(z)dz \\
&\quad+(-n)z_{0}^{-n-1}\lim_{\epsilon\searrow 0}\frac{1}{2\pi i} \int_{L_{\epsilon}}f^{\prime}(z)(z-z_{0})dz.
\end{aligned}
\end{equation}
For the integral $\int_{L_{\epsilon}}f^{\prime}(z)z^{-n}dz$, by (\ref{in+3}), we have
 \begin{equation}\label{(9)}
 \int_{L_{\epsilon}}f^{\prime}(z)dz=d_{-1}\int_{L_{\epsilon}}\frac{1}{(z-z_{0})^{2}}dz+\int_{L_{\epsilon}}g^{\prime}(z)dz.
 \end{equation}
 By Cauchy's theorem,
 \begin{equation*}
 \int_{L_{\epsilon}}\frac{1}{(z-z_{0})^{2}}+ \int_{\Omega_{\epsilon}}\frac{1}{(z-z_{0})^{2}}=0,\end{equation*}
 thus
 \begin{equation}\label{(6)+}
 \begin{aligned}
  \int_{L_{\epsilon}}\frac{1}{(z-z_{0})^{2}}&=-\int_{\Omega_{\epsilon}}\frac{1}{(z-z_{0})^{2}}\\
  &=-i\int_{\Omega_{\epsilon}} \frac{e^{it}}{(e^{it}-e^{it_{0}})^{2}}dt\\
  &=\frac{i}{e^{it_0}}\frac{1}{\tan(\epsilon/2)},
  \end{aligned}
  \end{equation}
  the last equality follows from (\ref{(6)}).
 Since $g^{\prime}(z)$ is analytic on the line $L_{\epsilon}$ for any $\epsilon>0$,
 we have \begin{equation}\label{(8)}
 \lim_{\epsilon\searrow 0}\int_{L_{\epsilon}}g^{\prime}(z)dz=0.
\end{equation}
 Substitute (\ref{(6)+}) and (\ref{(8)}) into (\ref{(9)}), we get
 \begin{equation}\label{(14)}
 \lim_{\epsilon\searrow 0} \int_{L_{\epsilon}}f^{\prime}(z)dz=\lim_{\epsilon\searrow 0}d_{-1}\frac{i}{e^{it_0}}\frac{1}{\tan(\epsilon/2)}.
  \end{equation}
  Now we consider the integral $\int_{L_{\epsilon}}f^{\prime}(z)(z-z_{0})dz$ in (\ref{(9)}).
  Again by (\ref{in+3}), we have
 \begin{equation}\label{(10)}
 \int_{L_{\epsilon}}f^{\prime}(z)(z-z_{0})dz=d_{-1}\int_{L_{\epsilon}}\frac{1}{z-z_{0}}dz+\int_{L_{\epsilon}}g^{\prime}(z)(z-z_{0})dz.
 \end{equation}
 Since $L_{\epsilon}$ be the line connected the points $e^{i(t_0-\epsilon)}$ and $e^{i(t_0+\epsilon)}$ and
 $$e^{it}=\cos(t)+i\sin(t),$$
 we have the parametrized equation
 \begin{equation*}
\begin{aligned}
L_{\epsilon}:=\left\{z=x+iy~~\biggl|~~\begin{matrix}
                                 x=\cos(t_0-\epsilon)+it\sin(t_0-\epsilon) \\
                                 y=\cos(t_0+\epsilon)+it\sin(t_0-\epsilon) \\
                                 0\leq t\leq 1
                               \end{matrix} \right\},
\end{aligned}
\end{equation*}
thus
\begin{equation}\label{(10)+}
\begin{aligned}
\int_{L_{\epsilon}}\frac{1}{z-z_{0}}dz&=\ln(z-z_{0})\bigg|_{z=e^{i(t_0-\epsilon)}}^{z=e^{i(t_0+\epsilon)}}\\
&=\ln(e^{i(t_0+\epsilon)}-e^{it_0})-\ln(e^{i(t_0-\epsilon)}-e^{it_0}).
\end{aligned}
\end{equation}
Note that
\begin{equation}\label{Gai+1}
\begin{aligned}
\ln(e^{i(t_0+\epsilon)}-e^{it_0})&=\ln|e^{i(t_0+\epsilon)}-e^{it_0}|+i\arg(e^{i(t_0+\epsilon)}-e^{it_0})\\
&=\ln|e^{i\epsilon}-1|+i\arg(e^{i(t_0+\epsilon)}-e^{it_0}).
\end{aligned}
\end{equation}
Since
\begin{equation*}
e^{i(t_0+\epsilon)}-e^{it_0}=e^{it_{0}}(e^{i\epsilon}-1),
\end{equation*}
we have \begin{equation}\label{Gai+2}
\begin{aligned}
\arg(e^{i(t_0+\epsilon)})-e^{it_0})&=\arg e^{it_{0}} +\arg(e^{i\epsilon}-1)\\
&=t_{0}+\arg(e^{i\epsilon}-1).
\end{aligned}
\end{equation}
Note that \begin{equation*}
\begin{aligned}
e^{i\epsilon}-1&=(\cos\epsilon-1)+i\sin\epsilon\\
&=-2\sin^{2}(\epsilon/2)+i2\sin(\epsilon/2)\cos(\epsilon/2),
\end{aligned}
\end{equation*}
we have
\begin{equation}\label{Gai+3}
\arg(e^{i\epsilon}-1)=-\tan^{-1}\left(\frac{\cos(\epsilon/2)}{\sin(\epsilon/2)}\right)=-\left(\frac{\pi}{2}-\frac{\epsilon}{2}\right).
\end{equation}
Combing (\ref{Gai+1}), (\ref{Gai+2}) and (\ref{Gai+3}), we get
\begin{equation}\label{(3.17)}
\ln(e^{i(t_0+\epsilon)}-e^{it_0})=\ln|e^{i\epsilon}-1|-i\left(\frac{\pi}{2}-t_{0}-\frac{\epsilon}{2}\right).
\end{equation}
Similarly, we have
\begin{equation}\label{(3.17)+}
\ln(e^{i(t_0-\epsilon)}-e^{it_0})=\ln|e^{i\epsilon}-1|+i\left(\frac{\pi}{2}+t_{0}-\frac{\epsilon}{2}\right).
\end{equation}
The substitute (\ref{(3.17)}) and (\ref{(3.17)+}) into (\ref{(10)+}), we get
\begin{equation*}\int_{L_{\epsilon}}\frac{1}{z-z_{0}}dz=-i(\pi-\epsilon)
\end{equation*}
and
\begin{equation}\label{(11)}
\lim_{\epsilon\searrow 0}\int_{L_{\epsilon}}\frac{1}{z-z_{0}}dz=-i\pi.
\end{equation}
Since $g^{\prime}(z)$ is analytic on the line $L_{\epsilon}$ for any $\epsilon>0$,
we have
\begin{equation}\label{(12)}
 \lim_{\epsilon\searrow 0}\int_{L_{\epsilon}}g^{\prime}(z)(z-z_{0})dz=0.
\end{equation}
Then substitute (\ref{(11)}) and (\ref{(12)}) into (\ref{(10)}), we get
\begin{equation}\label{(13)}
\lim_{\epsilon\searrow 0}\int_{L_{\epsilon}}f^{\prime}(z)(z-z_{0})dz=-d_{-1}i\pi.
\end{equation}
So substitute  (\ref{(14)}) and (\ref{(13)})  into (\ref{(7)}), we get
\begin{equation}\label{(15)}
\begin{aligned}
&\quad\lim_{\epsilon\searrow 0}\frac{1}{2\pi i}\int_{L_{\epsilon}}f^{\prime}(z)z^{-n}dz\\&=\frac{1}{2\pi}\left(\lim_{\epsilon\searrow 0}\frac{d_{-1}e^{-nit_0}}{e^{it_0}\tan(\epsilon/2)}+d_{-1}n\pi e^{-(n+1)it_{0}}\right),
\end{aligned}
 \end{equation}
and substitute the above result into (\ref{(4)}) we get
 \begin{equation}\label{(3.22)}
 \begin{aligned}
&\quad \lim_{\epsilon\searrow 0}\frac{1}{2\pi}\int_{\Omega_\epsilon}f^{\prime}(e^{it})e^{it}e^{-nit}dt
\\&=na_{n}-\frac{1}{2\pi}\left(\lim_{\epsilon\searrow 0}\frac{d_{-1}e^{-nit_0}}{e^{it_0}\tan(\epsilon/2)}+d_{-1}n\pi e^{-(n+1)it_{0}}\right). \end{aligned}
 \end{equation}
 If $n=0$, by (\ref{(4)}) and (\ref{(14)}) we have
 \begin{equation}
\lim_{\epsilon\searrow 0}\frac{1}{2\pi}\int_{\Omega_\epsilon}f^{\prime}(e^{it})e^{it}dt=-\lim_{\epsilon\searrow 0}\frac{1}{2\pi i}\int_{L_{\epsilon}}f^{\prime}(z)dz=-\frac{1}{2\pi}\lim_{\epsilon\searrow 0}\frac{d_{-1}}{e^{it_0}}\frac{1}{\tan(\epsilon/2)},
\end{equation}
thus (\ref{(3.22)}) is also applicable in this case.
 Finally, substitute (\ref{(3.22)}) into (\ref{fivestar}), we have the $n$th Fourier coefficient
 $c_n(D)=na_{n}$ for $n\in\mathbb{N}$, as desired.
  \end{proof}

\section{The Fourier series of $D$ is summable at $t=0$}
In this section, we consider the summable of the Fourier series of $D$ at $t=0$ and prove the following theorem.
\begin{theorem}\label{4.1}
For $k\in\mathbb{N}$, let $D^{(k-1)}$ be the $(k-1)$th derivative of $D$, that is,
$$D^{(k-1)}=i^{k-1}\left(1^{k}a_{1}e^{it}+2^{k}a_{2}e^{2it}+3^{k}a_{3}e^{3it}+\cdots+n^{k}a_{n}e^{nit}+\cdots\right)\in\calD^{\prime}(\mathbb{R}).$$
Then $D^{(k-1)}$ is summable at $t=0$ in the sense of distributions (see Def. \ref{Def 2.4}) and we have the sum
$$1^{k}a_{1}+2^{k}a_{2}+3^{k}a_{3}+\cdots+n^{k}a_{n}+\cdots=\frac{1}{i^{k}}\left(\frac{d}{dt}\right)^{k}f(e^{it})\bigg|_{t=0},$$
where $f(z)$ is the corresponding power series (see (\ref{in+A})).
\end{theorem}
\begin{proof}
In the case of $R>1$, the proof goes the same as \cite[Proposition 8.1]{Amol}.
So we modify this proof to suitable for $R=1$.
In this case, by our assumption, the singular point $z_{0}=e^{it_{0}}\neq 1$, that is $t_{0}\neq 0$.
Let $\{\varphi_m\}_{m=1}^{\infty}$ be any approximate
  identity. For all large enough $m$, the support of $\varphi_m$ is
  contained inside $(-\delta, \delta)$ for some small delta such that $t_{0}\not\in(-\delta,\delta)$. Then we have
\begin{eqnarray}
\nonumber
 \lim_{m\rightarrow \infty}
 \langle D^{(k-1)}, \varphi_m\rangle
 &=&
 \lim_{m\rightarrow \infty}
(-1)^{k-1} \langle D, \varphi_m^{(k-1)}\rangle
\\
\nonumber
&=&
 \lim_{m\rightarrow \infty}
(-1)^{k-1} \int_{-\delta}^{\delta}f^{\prime}(e^{it})e^{it} \varphi_m^{(k-1)}(t)dt\\
\nonumber
&\quad& (\textrm{by}~(\ref{defD})~\textrm{and~note~that}~ t_{0}\not\in(-\delta,\delta))\\
\nonumber
&=&
 \lim_{m\rightarrow \infty}
(-1)^{k-1}
\left\langle \varphi_m^{(k-1)}, f^{\prime}(e^{it})e^{it} \right\rangle
\\
\nonumber
&=&
 \lim_{m\rightarrow \infty}
\left\langle \varphi_m, \left(\frac{d}{dt}\right)^{k-1}f^{\prime}(e^{it})e^{it}\right\rangle
\\
\nonumber
&=&
\left\langle \delta_0, \left(\frac{d}{dt}\right)^{k-1}f^{\prime}(e^{it})e^{it}\right\rangle
\\
\nonumber
&=&
\left.\left(\frac{d}{dt}\right)^{k-1}f^{\prime}(e^{it})e^{it}\right|_{t=0}\\
&=&
\nonumber
\frac{1}{i}\left.\left(\frac{d}{dt}\right)^{k}f(e^{it})\right|_{t=0}.
\end{eqnarray}
This completes the proof.
\end{proof}

\section{Apllications}
In this section, as applications of our constructions, we give some examples.
\begin{example}\label{eg1}
For $k\in\mathbb{N},$ the alternating series is given by
$$A_{k}=1^{k}-2^{k}+3^{k}-\cdots+(-1)^{n-1}n^{k}+\cdots.$$
By Theorem \ref{4.1}, we consider the power series
$$f(z)=z-z^{2}+z^{3}-\cdots+(-1)^{n-1}z^{n}+\cdots$$
for $|z|<1$.
Since $$z-z^{2}+z^{3}-\cdots+(-1)^{n-1}z^{n}+\cdots=\frac{z}{z+1}$$
for $|z|<1$, $f(z)$ can be analytic continued to the complex plane as  a
meromorphic function $$f(z)=\frac{z}{z+1}$$ with a single pole $z_{0}=-1$ on the unit
circle $C_{1}$.
So by Theorem \ref{4.1}, we have
\begin{equation}\label{(17)}
\begin{aligned}
1^{k}-2^{k}+3^{k}-\cdots+(-1)^{n-1}n^{k}+\cdots&=\frac{1}{i^{k}}\left(\frac{d}{dt}\right)^{k}\frac{e^{it}}{e^{it}+1}\bigg|_{t=0}\\
&=\frac{1}{i^{k}}\left(\frac{d}{dt}\right)^{k}\left(1-\frac{1}{e^{it}+1}\right)\bigg|_{t=0}\\
&=-\frac{1}{i^{k}}\left(\frac{d}{dt}\right)^{k}\frac{1}{e^{it}+1}\bigg|_{t=0}.
\end{aligned}
\end{equation}
Recall that
Euler polynomials are defined by the generating function
\begin{equation*}
\frac{2e^{xt}}{e^{t}+1}=\sum_{m=0}^{\infty}E_{m}(x)\frac{t^m}{m!}
\end{equation*}
(see \cite{Kim2009}),
thus \begin{equation}\label{(18)}
\frac{2}{e^{t}+1}=\sum_{m=0}^{\infty}E_{m}(0)\frac{t^m}{m!}.
\end{equation}
Substitute into (\ref{(17)}), we have
\begin{equation}\label{(20)}
\begin{aligned}
A_{k}&=1^{k}-2^{k}+3^{k}-\cdots+(-1)^{n-1}n^{k}+\cdots\\
&=-\frac{1}{i^{k}}\left(\frac{d}{dt}\right)^{k}\frac{1}{e^{it}+1}\bigg|_{t=0}\\
&=-\frac{1}{2i^{k}}\left(\frac{d}{dt}\right)^{k}\frac{2}{e^{it}+1}\bigg|_{t=0}\\
&=-\frac{1}{2i^{k}}\left(\frac{d}{dt}\right)^{k}\left(\sum_{m=0}^{\infty}E_{m}(0)\frac{(it)^m}{m!}\right)\bigg|_{t=0}\\
&=-\frac{E_{k}(0)}{2}.
\end{aligned}
\end{equation}
Since $E_{1}(0)=-\frac{1}{2},$ in particular, we have
$$A=1-2+3-\cdots+(-1)^{n-1}n+\cdots=\frac{1}{4}$$
and by (\ref{(3)+})
$$N=1+2+3+\cdots+n+\cdots=-\frac{1}{3}A=-\frac{1}{12},$$
which is the same as (\ref{(1)+}) predicated by the Casimir effect.
\end{example}
\begin{example}
Furthermore, let the Bernoulli numbers $B_{m}$ be defined by the following generating function
\begin{equation}\label{(19)}
\frac{t}{e^{t}-1}=\sum_{m=0}^{\infty}B_{m}\frac{t^m}{m!}.
\end{equation}
Comparing the generating function (\ref{(18)} and (\ref{(19)}), we have a relation between $B_{k}$ and $E_{k}(0)$:
\begin{equation}\label{(20)+}
E_{k}(0)=2(1-2^{k+1})\frac{B_{k+1}}{k+1}
\end{equation}
for $k\in\mathbb{N}.$ Denote by
$$N_{k}=1^{k}+2^{k}+3^{k}-\cdots+n^{k}+\cdots.$$
It can be seen that
\begin{equation}
\begin{aligned}
A_k&=(1^k+2^k+3^k+4^k+\cdots)-2(2^k+4^k+6^k+8^k+\cdots)\\
&=(1^k+2^k+3^k+4^k+\cdots)-2^{k+1}(1^k+2^k+3^k+4^k+\cdots)\\
&=(1-2^{k+1})N_{k},
\end{aligned}
 \end{equation}
thus $$N_{k}=\frac{1}{1-2^{k+1}}A_{k}.$$
As mentioned above, this was part of the definition of the summation method from \cite{Amol}, via a homothetic transformation which doubled the period of a periodic distribution. 
Then substitute  (\ref{(20)}) to the above equality and notice  (\ref{(20)+}), we have
\begin{equation}\label{(21)}
N_{k}=\frac{1}{1-2^{k+1}}\cdot\left(-\frac{1}{2}E_{k}(0)\right)=-\frac{B_{k+1}}{k+1}.
\end{equation}
This result was first obtained by Euler in 1740 (see \cite[p. 203]{Stopple}).
\end{example}
\begin{example}
For $k\in\mathbb{N}$ and a complex number $\epsilon$ with $|\epsilon|\leq1$, but $\epsilon\neq 1$, we consider the sum
\begin{equation*}
\epsilon^{1}1^{k}+\epsilon^{2} 2^{k}+\epsilon^{3} 3^{k}+\cdots+\epsilon^{n}n^{k}+\cdots.
\end{equation*}
Then the  corresponding power series is
$$f(z)=\epsilon z+\epsilon^{2}z^{2}+\cdots+\epsilon^{n}z^{n}+\cdots$$
for $|z|<1.$
It is easy to see that $$f(z)=\sum_{n=1}^{\infty}\epsilon^{n}z^{n}=\frac{1}{1-\epsilon z}-1$$
for $|z|<1$, thus it satisfies the conditions of Section \ref{Sec.1}.
Then by Theorem \ref{4.1}, we have
\begin{equation}\label{(16)}
\begin{aligned}
\epsilon^{1} 1^{k}+\epsilon^{2} 2^{k}+\epsilon^{3} 3^{k}+\cdots+\epsilon^{n}n^{k}+\cdots
&=\frac{1}{i^{k}}\left(\frac{d}{dt}\right)^{k}\left(\frac{1}{1-\epsilon e^{it}}-1\right)\bigg|_{t=0}\\
&=\frac{1}{i^{k}}\left(\frac{d}{dt}\right)^{k}\left(\frac{1}{1-\epsilon e^{it}}\right)\bigg|_{t=0}.
\end{aligned}
\end{equation}
Recall that the Apostol--Bernoulli numbers $B_{m}(\epsilon)$ are
defined by the following generating function
\begin{equation}
\begin{aligned}
\frac{t}{\epsilon e^{t}-1}&=\sum_{m=0}^{\infty}B_{m}(\epsilon)\frac{t^m}{m!}\\
&=\sum_{m=1}^{\infty}B_{m}(\epsilon)\frac{t^m}{m!}\\
&\quad(\textrm{since~by~our~assumption}~\epsilon\neq 1,~\textrm{we~have}~B_{0}(\epsilon)=0)
\end{aligned}
\end{equation}
(see \cite[Eq. (3.1)]{Apostol}) or \cite[Eq. (1.3)]{HuKim2018}).
So substitute into (\ref{(16)}), we have
\begin{equation}
\begin{aligned}
&\quad\epsilon^{1}1^{k}+\epsilon^{2} 2^{k}+\epsilon^{3} 3^{k}+\cdots+\epsilon^{n}n^{k}+\cdots\\
&=\frac{1}{i^{k}}\left(\frac{d}{dt}\right)^{k}\left(\frac{1}{1-\epsilon e^{it}}\right)\bigg|_{t=0}\\
&=-\frac{1}{i^{k}}\left(\frac{d}{dt}\right)^{k}(it)^{-1}\left(\frac{it}{\epsilon e^{it}-1}\right)\bigg|_{t=0}\\
&=-\frac{1}{i^{k}}\left(\frac{d}{dt}\right)^{k}(it)^{-1}\left(\sum_{m=1}^{\infty}B_{m}(\epsilon)\frac{(it)^m}{m!}\right)\bigg|_{t=0}\\
&=-\frac{1}{i^{k}}\left(\frac{d}{dt}\right)^{k}\sum_{m=0}^{\infty}B_{m+1}(\epsilon)\frac{(it)^m}{(m+1)!}\bigg|_{t=0}\\
&=-\frac{B_{k+1}(\epsilon)}{k+1}.
\end{aligned}
\end{equation}
\end{example}

\textbf{Acknowledgements.} We would like to thank Professor Christian Schubert for his interested in this work and for his helpful comments.

 \bibliography{central}

\end{document}